\documentclass[12pt,a4paper]{amsart}

\usepackage{amsmath,amsfonts,amssymb,amsthm,mathrsfs,microtype}
\usepackage{hyperref,graphicx}
\usepackage{cite}
\usepackage{color}

\makeatletter

\theoremstyle{plain} \newtheorem{theorem}{Theorem}[section]
\newtheorem{lemma}[theorem]{Lemma}

\theoremstyle{definition}
\newtheorem{definition}[theorem]{Definition}
\newtheorem{example}[theorem]{Example}
\newtheorem{remark}[theorem]{Remark}
\newtheorem{notation}[theorem]{Notation}

\DeclareMathOperator{\dimh}{dim_H}

\DeclareMathOperator{\T}{\mathbb{T}}

 \date{}

\makeatother

\allowdisplaybreaks[3]

\begin{document}
\title{On recurrence sets for toral endomorphisms }


\author{Zhang-nan Hu*}\thanks{* Corresponding author}  \address{Z.-N.~Hu, College of Science,
  China University of Petroleum, Beijing 102249, China}
\email{hnlgdxhzn@163.com}

\author{Tomas Persson} \address{T.~Persson, Centre for
  Mathematical Sciences, Lund University, Box~118, 221~00~Lund,
  Sweden} \email{tomasp@gmx.com}

\date{\today}

\begin{abstract}
  Let $A$ be a $2\times 2$ integral matrix with an eigenvalue of
  modulus strictly less than 1. Let $T$ be the natural
  endomorphism on the torus
  $\mathbb{T}^2=\mathbb{R}^2/\mathbb{Z}^2$, induced by $A$. Given
  $\tau>0$, let
  \[
    R_\tau =\{\, x\in \mathbb{T}^2 : T^nx\in
    B(x,e^{-n\tau})~\mathrm{infinitely ~many}~n\in\mathbb{N}
    \,\}.
  \]
  We calculated the Hausdorff dimension of $R_\tau$, and also
  prove that $R_\tau$ has a large intersection property.
\end{abstract}

\maketitle

\section{Introduction}

Let $(X, \mathscr{B}, \mu,
T)$ be a measure preserving system equipped with a compatible
metric $\rho$. Let $B(x,r)$ be a ball in $X$ with center
$x$ and radius $r>0$.  If
$(X,\rho)$ is separable, the Poincar\'e Recurrence Theorem yields
that the set
\[
  \{\, x\in X : T^n(x)\in B(x,r) ~{\rm for ~infinitely ~many
    ~}n\in\mathbb{N} \,\}
\]
has full $\mu$-measure. It is qualitative in nature, and
considering the speed of recurrence leads to the study of the
\emph{recurrence set}
\[
  R_\tau = \bigl\{\, x\in X : T^n(x)\in B(x,\psi(n)) ~{\rm for
    ~infinitely ~many ~}n\in \mathbb{N} \,\bigr\},
\]
where $\psi \colon \mathbb{N} \to (0, \infty)$ is a positive
function.  Establishing measure and dimension results for
$R_\tau$ is a fundamental problem. An important result on the
measure of $R_\tau$ is due to Boshernitzan \cite{Bo} who proved
that if $\mu$ is a probability measure, and the
$\alpha$-dimensional Hausdorff measure $\mathcal{H}^\alpha$ of
$X$ is $\sigma$-finite, then for $\mu$-almost all $x\in X$
\[
  \liminf_{n\to\infty}n^{1/\alpha}\rho(T^n(x),x)<\infty.
\]
Moreover, if $\mathcal{H}^\alpha(X)=0$, he also proved that
\[
  \liminf_{n\to\infty}n^{1/\alpha}\rho(T^n(x),x)=0.
\]
This result was further developed by Barreira and Saussol
\cite{BS}. They showed that the exponent $\alpha$ could be
replaced by the lower local dimension of $\mu$ at $x$. If the
measure $\mu$ is uniform, it is shown that $\mu(R_\tau)$ is
related to the convergence or divergence of series
$\sum_{n=1}^\infty \psi(n)^t$ for some $t>0$. We refer to
\cite{bafa,bako,cww,heliao,hlsw,kkp,kz} for some recent
results. Kirsebom, Kunde and Persson \cite{kkp} considered the
case where $\mu$ is nonuniform.

If the measure of $R_\tau$ is 0, then those measure statements
provide little information regarding the `size', hence it is
interesting to determine the dimension of the set to describe the
size of $R_\tau$.  It was first studied by Tan and Wang
\cite{tanwang} for $\beta$-transformations. Later, results have
been obtained for certain expanding and hyperbolic systems on
tori, which we shall now describe.

Let $A$ be a real $d\times d$ matrix.  Then $A$ determines a
self-map $T$ of the torus
$\mathbb{T}^d := \mathbb{R}^d/\mathbb{Z}^d$, defined by
\[
  T(x)= Ax \pmod 1.
\]
Our recurrence set is then
\[
  R_\tau = \bigl\{\, x\in\T^2 : T^n(x)\in B(x,e^{-n\tau})~{\rm
    for ~infinitely ~many ~}n\in \mathbb{N} \,\bigr\}.
\]
He and Liao \cite{heliao} gave a dimension formula for $R_\tau$
when $A$ is a diagonal expanding matrix. Later Hu and Li
\cite{huli} generalised this result to expanding integer matrices
under some condition.

Many other expanding dynamical systems have also been addressed
(see \cite{sw,wy} and references within). Our previous paper
\cite{hupe} investigated the linear toral automorphisms of
$\mathbb{T}^2$ in the case when $\det A = \pm 1$ and one
eigenvalue is outside the unit circle. Under these assumptions,
we gave a formula for the Hausdorff dimension of $R_\tau$ in
terms of $\tau$ and the eigenvalues of $A$. The main purpose of
this paper is to extend our previous result on the Hausdorff
dimension to $R_\tau$ \cite{hupe} to a wider set of matrices $A$.
More precisely, we extend our dimension formula for $R_\tau$ to
any invertible hyperbolic $2\times 2$ integer matrix, that is,
integer matrices with one eigenvalue outside the unit cicle and
one non-zero eigenvalue inside the unit circle.

The recurrence set $R_\tau$ is a limsup set. Falconer
\cite{falconer1} introduced the class of sets with large
intersection property describing the size of limsup sets in sense
of Hausdorff dimension, which is denoted by
$\mathscr{G}^s(\mathbb{R}^d)$. Since in this paper, we study
subsets of the torus $\mathbb{T}^d$, we consider instead the
class $\mathscr{G}^s(\mathbb{T}^d)$ whose definition is
straightforward by restricting sets in
$\mathscr{G}^s(\mathbb{R}^d)$ to $\mathbb{T}^d$. The method used
in this paper also proves that the recurrence set $R_\tau$ has a
large intersection property.



\begin{theorem}
  \label{main}
  Let $A$ be a $2 \times 2$ integer matrix with eigenvalues
  $|\lambda_2|>1>|\lambda_1|$. Let $T(x)=Ax \pmod 1$ and
  $\tau>0$. Then
  \[
    \dimh R_\tau = s_0,
  \]
  where
  \[
    s_0= \min \Bigl\{\frac{2\log|\lambda_2|}{\tau +
      \log|\lambda_2|}, \frac{\log|\lambda_2|}{\tau} \Bigr\}.
  \]
  Moreover, for $\tau>0$, we have
  $R_\tau\in\mathcal{G}^{s_0} (\T^2).$
\end{theorem}

We remark that the dimension formula above is same as the one
given in \cite{hupe}.  When $d>2$, the dimension formula for
$\dimh R_\tau$ may be more complicated as the following example
shows. This example draws inspiration from
\cite[Example~7.2]{HPWZ}.




\begin{example}\label{ex:firstexample}
  Let $A=\begin{bmatrix} m&0&0\\0&a&b\\0&c&d\end{bmatrix}$ be an
  integral matrix and $m>1$. Assume that
  $\begin{bmatrix} a&b\\c&d\end{bmatrix}$ satisfies $ad-bc=1$
  with an eigenvalue $\lambda$ strictly larger than 1. We also
  suppose that
  $m>\lambda^{1/2}$. 
  \begin{itemize}
  \item If $m>\lambda$, then
    \begin{multline*}
      \dimh R_\tau=\min\Bigl\{\frac{\tau+3\log \lambda}{\tau+\log
        \lambda},\frac{3\log m}{\tau+\log m},\frac{2\log
        \lambda+\log m}{\tau+\log
        \lambda},\\
      \frac{\log\lambda+\log m}{\tau},\frac{\tau+\log
        \lambda}{\tau}\Bigr\}.
    \end{multline*}

  \item If $m\le \lambda$, then
    \[
      \dimh R_\tau=\min\Bigl\{\frac{2\log m+\log
        \lambda}{\tau+\log m},\frac{3\log \lambda}{\tau+\log
        \lambda},\frac{\log\lambda+\log m}{\tau},\frac{\tau+\log
        \lambda}{\tau}\Bigr\}.
    \]
  \end{itemize}

  By \cite[Lemma~5.1]{HPWZ}, $\lambda$ is irrational, and so
  $m\ne\lambda^{1/2}$. The case when $m<\lambda^{1/2}$ is as
  complex as when $m>\lambda^{1/2}$. One can use a similar method
  as for Example~\ref{ex:firstexample} to study the case that $m<\lambda^{1/2}$, but we will not do so here. The proof of Example~\ref{ex:firstexample} will be given in Section~\ref{sec:example}.
\end{example}

\begin{remark}
  Let $A$ be a $d \times d$ hyperbolic integer matrix. Suppose
  that the dimension formula for $R_\tau$ depends only on the
  eigenvalues of $A$. Let $\ell_i$ denote the logarithms of the
  eigenvalues of $A$, ordered such that $\ell_i \leq
  \ell_{i+1}$. When $d=2,$ Theorem~\ref{main} shows that
  \[
    \min_{1\le i\le d}
    \biggl\{\frac{i\ell_i+\sum_{j=i+1}^d\ell_j}{\tau+\ell_i}
    \biggr\}
  \]
  is the Hausdorff dimension of $R_\tau$. It would therefore not
  be unnatural to expect that the above expression also gives the
  Hausdorff dimension of $R_\tau$ for $d > 2$.  However, when
  $d>2$, the above examples show that this is only an upper bound
  on the Hausdorff dimension, not the dimension formula.
\end{remark}






\section{Prelimimaries}
In this section we give some notation and technical results which
we will use in the proofs of Theorem~\ref{main}.

\begin{notation}
  In this paper, we will use the following notation.  Let $f_n$
  and $g_n$ be two sequences of real numbers. Write
  $f_n\lesssim g_n$ if and only if there exist constants $N\ge 1$
  and $c>1$ such that $f_n\le cg_n$ for all $n\ge N$, and
  $f_n\gtrsim g_n$ means $g_n\lesssim f_n$. $O(f_n)$ denotes some
  quality $\lesssim f_n$.  Write $f_n\asymp g_n$ if
  $f_n\lesssim g_n$ and $g_n\lesssim f_n$.
\end{notation}
 
For $n\ge1$, put
\[
  R_n = \bigl\{\, x\in\T^2 : T^nx\in B(x,e^{-n\tau}) \, \bigr\},
\]
then $R_\tau=\limsup\limits_{n\to\infty}R_n$.

\subsection{Geometric structure of $\boldsymbol{R_n}$}

For $n\ge1$, put
\[
  \mathcal{P}_n := \{\, x\in\T^2 : (A^n-I)x\pmod1=0 \,\} = (A^n -
  I)^{-1} \mathbb{Z}^2 \pmod 1.
\]
Then we have by \cite[Lemma~2.3]{EW} that
\[
  \#\mathcal{P}_n=\prod_{j=1}^2|\lambda_j^n-1|=:H_n.
\]
If $e^{-\tau n} < \frac{1}{2}$, then the set
\begin{equation*}
  \begin{split}
    R_n&=\{\, x\in\T^d : T^nx\in B(x,e^{-n\tau}) \,\}\\
    &=\bigcup_{x\in
      \mathcal{P}_n}(A^n-I)^{-1}B(0,e^{-n\tau})+\{x\}
  \end{split}
\end{equation*}
consists of $H_n$ disjoint ellipses, denoted by $\{R_{n,i}\}_i$,
whose centers are the points in $\mathcal{P}_n$.

In the following, we will investigate the geometric structure of
such ellipses.  Now we give some notation.

Assume that
$A= \bigl[ \begin{smallmatrix} a & b \\ c & d \end{smallmatrix}
\bigr]$ with eigenvalues $|\lambda_2|>1>|\lambda_1|$.
Denote $\gamma=\frac{\lambda_2-a}{b}$ and
$\beta=\frac{\lambda_1-a}{b}$, then
\[
  \begin{bmatrix} 1\\ \gamma
  \end{bmatrix} \quad{\rm
    and} \quad  \begin{bmatrix} 1\\
    \beta
  \end{bmatrix}
\]
are eigenvectors with eigenvalues $\lambda_2$ and $\lambda_1$
respectively. 

For $A$ we write
\begin{align*}
  \lambda_{n,1}
  = \frac{e^{-\tau n}}{|1 - \lambda_1^{n}|}
  \quad {\rm and}\quad 
  \lambda_{n,2}  
  =\frac{e^{-\tau n}}{|\lambda_2^n - 1|},
\end{align*} 
and let $e_{n,2} > e_{n,1}$ be the length of semi-axes of the
ellipse $R_{n,i}$.

The following lemma shows that each ellipse $R_{n,i}$ contains a
parallelogram and is contained in a parallelogram. Both
parallelograms are comparable in size as
Lemma~\ref{lem:parallelograms} and Figure~\ref{fig:parallelogram}
show.

\begin{lemma}
  \label{lem:parallelograms}
  For $i\ge1$, the ellipse $R_{n,i}$ contains an inscribed
  parallelogram $E_{n,i}$ with vertices
  \begin{align*}
    &x_{n,i} +\frac{1}{2}\bigl(\frac{\lambda_{n,2}}{\sqrt{1 +
      \gamma^2}}
      \begin{bmatrix}
        1 \\ \gamma
      \end{bmatrix}
    \pm \frac{\lambda_{n,1}}{\sqrt{1 + \beta^2}}
    \begin{bmatrix}
      1 \\ \beta
    \end{bmatrix} \bigr), \\
    &x_{n,i}+\frac{1}{2}\bigl(-\frac{\lambda_{n,2}}{\sqrt{1 +
      \gamma^2}}
      \begin{bmatrix}
        1 \\ \gamma
      \end{bmatrix}
    \pm \frac{\lambda_{n,1}}{\sqrt{1 + \beta^2}}
    \begin{bmatrix}
      1 \\ \beta
    \end{bmatrix}\bigr),
  \end{align*}
  whose centre $x_{n,i}$ is an $n$-periodic
  point. 
  The side lengths of $E_{n,i}$ are $ \lambda_{n,1}$
  and $ \lambda_{n,2}$, and the ellipse $R_{n,i}$ is contained in a
  bigger parallelogram $\tilde{E}_{n,i}$ with vertices
  \begin{align*}
    & x_{n,i} +c_1\bigl(\frac{\lambda_{n,2}}{\sqrt{1 + \gamma^2}} \begin{bmatrix} 1 \\ \gamma
    \end{bmatrix} \pm \frac{\lambda_{n,1}}{\sqrt{1 + \beta^2}} \begin{bmatrix} 1 \\
      \beta \end{bmatrix} \bigr),
    \\
    & x_{n,i}+c_1\bigl(-\frac{\lambda_{n,2}}{\sqrt{1 +
      \gamma^2}} \begin{bmatrix} 1 \\ \gamma \end{bmatrix} \pm
    \frac{\lambda_{n,1}}{\sqrt{1 + \beta^2}} \begin{bmatrix}
      1 \\
      \beta \end{bmatrix}\bigr) ,
  \end{align*}
  where $c_1>1$ only depends on $A$.
\end{lemma}
  
\begin{figure}
  \begin{center}
    \includegraphics[scale=0.3]{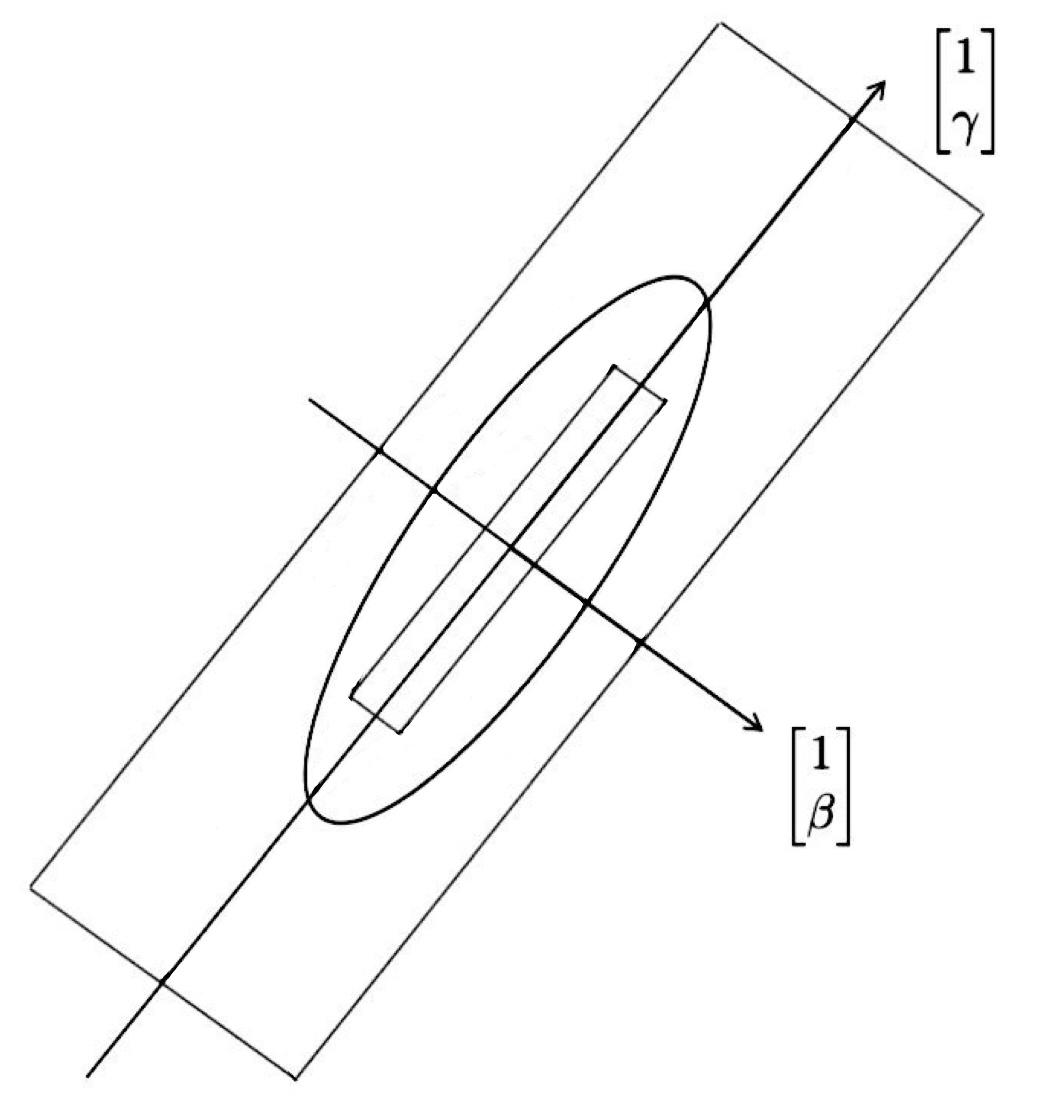}
  \end{center}
  \caption{The ellipse $R_{n,i}$ and parallelograms $E_{n,i}$ and
    $\tilde{E}_{n,i}$ which are marked in orange for some $n$.}
  \label{fig:parallelogram}
\end{figure}

\begin{proof}
  Since each $R_{n,i}$ is a translation of the ellipse
  $(A^n-I)^{-1} B(0,e^{-n\tau})$, we can inscribe a parallelogram
  in $R_{n,i}$ as follows. Consider the points
  \begin{align*}
    &\frac{ e^{-n\tau}}{2}
      \Bigl(\frac{1}{\sqrt{1+\gamma^2}}\begin{bmatrix} 1 \\
        \gamma \end{bmatrix}\pm
    \frac{1}{\sqrt{1+\beta^2}} \begin{bmatrix} 1 \\
      \beta \end{bmatrix}\Bigr), \\
    &\frac{ e^{-n\tau}}{2}
      \Bigl(-\frac{1}{\sqrt{1+\gamma^2}}\begin{bmatrix} 1 \\
        \gamma \end{bmatrix}\pm
    \frac{1}{\sqrt{1+\beta^2}} \begin{bmatrix} 1 \\
      \beta \end{bmatrix}\Bigr)
  \end{align*}
  which are points in $B(0,e^{-n\tau})$.  Since the vectors
  $\begin{bmatrix} 1 \\ \gamma \end{bmatrix}$ and
  $\begin{bmatrix} 1 \\ \beta \end{bmatrix}$ are eigenvectors of
  $A$, it follows that the ellipse $(A^n-I)^{-1} B(0,e^{-n\tau})$
  has an inscribed parallelogram with vertices given by
  \begin{align*}
    & \frac{1}{2}\Bigl(\frac{\lambda_{n,2}}{\sqrt{1 +
      \gamma^2}} \begin{bmatrix} 1 \\ \gamma 
  \end{bmatrix} \pm \frac{\lambda_{n,1}}{\sqrt{1 +
    \beta^2}} \begin{bmatrix} 1 \\
  \beta \end{bmatrix} \Bigr), \\
    & \frac{1}{2}\Bigl(-\frac{\lambda_{n,2}}{\sqrt{1 +
      \gamma^2}} \begin{bmatrix} 1 \\ \gamma \end{bmatrix} \pm
    \frac{\lambda_{n,1}}{\sqrt{1 +
    \beta^2}} \begin{bmatrix} 1 \\
  \beta \end{bmatrix}\Bigr) .
  \end{align*}
  
  In the same
  way, 
  the ellipse $(A^n-I)^{-1} B(0,e^{-n\tau})$ can be contained in
  a larger parallelogram with vertices given by
  \begin{align*}
    & c_1\Bigl(\frac{\lambda_{n,2}}{\sqrt{1 +
      \gamma^2}} \begin{bmatrix} 1 \\ \gamma 
  \end{bmatrix} \pm \frac{\lambda_{n,1}}{\sqrt{1 +
    \beta^2}} \begin{bmatrix} 1 \\
  \beta \end{bmatrix} \Bigr), \\
    & c_1\Bigl(-\frac{\lambda_{n,2}}{\sqrt{1 +
      \gamma^2}} \begin{bmatrix} 1 \\ \gamma \end{bmatrix} \pm
    \frac{\lambda_{n,1}}{\sqrt{1 +
    \beta^2}} \begin{bmatrix} 1 \\
  \beta \end{bmatrix}\Bigr)
  \end{align*}
  with
  $c_1=\frac{\sqrt{1+\beta^2}\sqrt{1+\gamma^2}}{|\beta-\gamma|}$.

  Therefore
  the quotients $e_{n,i}/\lambda_{n,i}$ are bounded from above
  and from below by constants independent of $n$ and $i$.
\end{proof}
  

In the following two sections, we will give some definitions and
notation which are used in Section~\ref{sec:lowerbound}.

\subsection{Large intersection property}

The large intersection property was introduced by Falconer
\cite{falconer1}. The $s$-dimensional class of sets with large
intersections is denoted by $\mathscr{G}^s(\mathbb{R}^d)$, which
consists of sets of Hausdorff dimension at least $s$.

For $s>0$, the class $\mathscr{G}^s(\mathbb{R}^d)$ is defined as
the class of $G_\delta$-sets $F$ of $\mathbb{R}^d$ satisfying
\[
  \dimh \Bigl(\bigcap_{i=1}^\infty f_i(F)\Bigr)\ge s.
\]
for any countable sequence $(f_i)_{i=1}^\infty$ of similarities,
and closed under countable intersections. See also
\cite{falconer1,falconer2} for other ways to define
$\mathscr{G}^s(\mathbb{R}^d)$. It is also straightforward to
define $\mathscr{G}^s(\mathbb{T}^d)$ by restricting sets in
$\mathscr{G}^s(\mathbb{R}^d)$ to $\mathbb{T}^d$.

Persson and Reeve \cite{PR} developed a method for analysing the
Hausdorff dimension of limsup sets.  They related the class of
sets with large intersection property to some potentials and
energies. Based on this, there are several developments. We use a
simplified variation of Theorem~1.1 in \cite{PR} from
\cite{HPWZ}, which is important in estimating the lower bound on
$\dimh R_\tau$.

\begin{lemma}[Lemma~3.2 of \cite{HPWZ}]\label{lip}
  Let $E_n$ be open sets in $\mathbb{T}^d$ and let $\mu_n$ be
  probability measures with
  $\mu_n(\mathbb{T}^d \setminus E_n) = 0$.  Suppose that there
  are constants $C$ and $s$ such that
  \begin{equation}\label{munballn}
    C^{-1} \le \liminf_{n\to\infty} \frac{\mu_n(B)}{\mu(B)}\le
    \limsup_{n\to\infty} \frac{\mu_n(B)}{\mu(B)}  \le C
  \end{equation}
  for any ball $B$ and $\mu_n(B) \le Cr^s$ for all $n$ and any
  ball $B$ of radius $r$.  Then
  $\limsup\limits_{n\to\infty}
  E_n\in\mathscr{G}^s(\mathbb{T}^d).$
\end{lemma}

\subsection{Uniform distribution modulo 1}

For a real number $x$, let $[x]$ denote the integral part of $x$,
that is, the greatest integer not larger than $x$; let
$\{x\} = x - [x]$ be the fractional part of $x$. In this paper,
we also use the notation $\{ \ldots \}$ for sets. However the
context makes the meaning clear.

Let $(w_n)_{n\in\mathbb{N}}$ be a given sequence of real
numbers. For an integer $N\ge 1$ and a subset $E$ of $\T$, the
counting function $A(E; N; (w_n))$ is defined as the number of
terms $w_n,~ 1 \le n \le N,$ for which $\{w_n\}\in E$, that is
\[
  A(E; N; (w_n)) = \#\{\, 1 \le n \le N : \{w_n\}\in E \,\}.
\]

\begin{definition}
  The sequence $(w_n)_{n\in\mathbb{N}}$ of real numbers is
  \emph{uniformly distributed modulo 1} if for every pair $a$,
  $b$ of real numbers with $0 \le a < b < 1$, we have
  \[
    \lim_{N\to\infty}\frac{A([a,b); N; (w_n))}{N}= b - a.
  \]
\end{definition}

To estimate the lower bound on $\dimh R_\tau$, according to
Lemma~\ref{lip}, we need to estimate $\mu_n(B)$ for any ball $B$
and for the measures $(\mu_n)_n$ (to be defined later). The
approach in our paper depends on the following theorem, and we
refer to \cite{Bugeaud, KN} for more relevant background and
details.

\begin{theorem}[Theorem~1.3 in \cite{Bugeaud} or Example~2.1 in
  \cite{KN}]
  \label{udmo}
  For any irrational real number $\alpha$, the sequence
  $(n\alpha)_{n\ge1}$ is uniformly distributed modulo one.
\end{theorem}

\section{The upper bound on $\dimh R_\tau$}
Recall that
\[
  R_n = \{ x\in\mathbb{T}^2 : (A^n-I)x\pmod1\in B(0,e^{-n\tau})
  \,\} = \bigcup_{i=1}^{H_n} R_{n,i},
\]
and by Lemma~\ref{lem:parallelograms}, the set $R_{n,i}$ is an
ellipse with length of semi-axes comparable to
$e^{-n\tau}|\lambda_2^n-1|^{-1}$ and
$e^{-n\tau}|\lambda_1^n-1|^{-1}$.

For $N\ge 1$,
\[
  R_\tau\subset
  \bigcup_{n=N}^{\infty}R_n=\bigcup_{n=N}^{\infty}\bigcup_{i=1}^{H_n}R_{n,i}.
\]

There are two kinds of natural coverings of $R_n$ which gives
different upper bounds.

\begin{itemize}
\item[(i)] We use balls of radius
  $e^{-n\tau}|\lambda_1^n-1|^{-1}$ to cover $R_n$, the number of
  which is
  \[
    \lesssim H_n.
  \]
  Then
  \begin{equation*}
    \begin{split}
      \mathcal{H}^s(R_\tau)&\le \liminf_{N\to\infty} \sum_{n=N}^\infty \sum_{i=1}^{H_n}|R_{n,i}|^s\\
      &\lesssim \liminf_{N\to\infty} \sum_{n=N}^\infty H_ne^{-ns\tau}|\lambda_1^n-1|^{-s}\\
      &\lesssim \liminf_{N\to\infty} \sum_{n=N}^\infty e^{-n(\tau s-\log|\lambda_2|)}\\
      &=0,
    \end{split}
  \end{equation*}
  where the last equality holds for
  $s>\frac{\log|\lambda_2|}{\tau}$.

\item[(ii)] We use balls of radius
  $e^{-n\tau}|\lambda_2^n-1|^{-1}$ to cover $R_{n,i}$,
  $1\le i\le H_n$, and the number of such balls is
  \[
    \lesssim
    \frac{e^{-n\tau}|\lambda_1^n-1|^{-1}}{e^{-n\tau}|\lambda_2^n-1|^{-1}}=
    \Big|\frac{\lambda_2^n-1}{1-\lambda_1^n} \Big|.
  \]
  Then for $s>\frac{2\log|\lambda_2|}{\tau+\log|\lambda_2|}$,
  \begin{equation*}
    \begin{split}
      \mathcal{H}^s(R_\tau)&\lesssim  \liminf_{N\to\infty}  \sum_{n=N}^\infty \sum_{i=1}^{H_n}  \Big|\frac{\lambda_2^n-1}{1-\lambda_1^n} \Big|e^{-ns\tau}|\lambda_2^n-1|^{-s}\\
      &\lesssim \liminf_{N\to\infty} \sum_{n=N}^\infty e^{-n((\tau+\log|\lambda_2|)s-2\log|\lambda_2|)}\\
      &=0.
    \end{split}
  \end{equation*}
  Therefore we conclude that
  \[
    \dimh R_\tau\le
    \min\Bigl\{\frac{2\log|\lambda_2|}{\tau+\log|\lambda_2|},
    \frac{\log|\lambda_2|}{\tau}\Bigr\}.
  \]
\end{itemize}

\section{The lower bound on $\dimh R_\tau$}
\label{sec:lowerbound}

By Lemma~\ref{lem:parallelograms}, each ellipse $R_{n,i}$
contains the parallelogram $E_{n,i}$ with side lengths
\[
  \frac{e^{-n\tau}|\lambda_1^n-1|^{-1}}{\sqrt{1+\gamma^2}},\quad
  \frac{e^{-n\tau}|\lambda_2^n-1|^{-1}}{\sqrt{1+\beta^2}}.
\]
Put
\[
  E_n=\bigcup_{i=1}^{H_n}E_{n,i}\subset R_n,
\]
then $\limsup\limits_{n\to\infty}E_n\subset R_\tau$. We will give
a lower bound on $\dimh (\limsup\limits_{n\to\infty}E_n)$.

Without loss of generality, we assume that
$\lambda_2>\lambda_1>0$. Because if one of them is negative, we
may consider $A^{2n}$ instead of $A^n$, and
$\limsup\limits_{n\to\infty}R_{2n}\subset R_\tau$.

Let $\mu$ denote the 2-dimensional normalized Lebesgue measure on
$\T^2$. Define
\[
  \mu_n=\frac{1}{\mu(E_n)}\mu|_{E_n},
\]
where
\[
  \mu(E_n)=\frac{|\gamma-\beta|}{\sqrt{(1+\beta^2) (1+\gamma^2)}}
  e^{-2n\tau} \asymp e^{-2n\tau},
\]
and
\[
  \mu(E_{n,i}) = \frac{|\gamma-\beta|}{\sqrt{(1+\beta^2)
      (1+\gamma^2)}} \frac{e^{-2n\tau}}{H_n} \asymp
  \frac{e^{-2n\tau}}{H_n}.
\]

We will use Lemma~\ref{lip} to show the lower bound, and the
proof is divided into two parts.

\subsection{Limit behaviour of the measures $\boldsymbol{\mu_n}$}

For any ball $B\subset \T^2$, we shall show there exists a
constant $C>1$ such that
\[
  C^{-1}\le
  \liminf_{n\to\infty}\frac{\mu_n(B)}{\mu(B)}\le\limsup_{n\to\infty}\frac{\mu_n(B)}{\mu(B)}\le
  C.
\]

\begin{lemma}\label{number}
  Given $\tilde{B}=B(\tilde{x},\tilde{r})\subset \mathbb{T}^2$,
  there exists $N(\tilde{B})\ge 1$ such that for all
  $n\ge N(\tilde{B})$, we have
  \[
    \# \tilde{B}\cap \mathcal{P}_n\asymp \tilde{r}^2H_n.
  \]
\end{lemma}

\begin{proof}
  It suffices to estimate
  $\#((A^n-I)\tilde{B}) \cap \mathbb{Z}^2$.
  Note that there exist constants $M_1,M_2>0$, depending on $A$
  only, such that for $n$ large enough $(A^n-I)\tilde{B}$
  contains a parallelogram
  $\tilde{E}_{n,1}$ 
  with centre $x_n=(A^n-I)x_{\tilde{B}}$ and the vertices
  \[
    x_n +\tilde{r}\Bigl( M_1(\lambda_2^n-1) \begin{bmatrix} 1 \\
      \gamma
    \end{bmatrix} \pm M_2(1-\lambda_1^n) \begin{bmatrix} 0 \\ 1
    \end{bmatrix} \Bigr), \]
  \[
    x_n -\tilde{r}\Bigl(M_1(\lambda_2^n-1)\begin{bmatrix} 1 \\
      \gamma
    \end{bmatrix} \pm M_2(1-\lambda_1^n) \begin{bmatrix} 0 \\
      1 \end{bmatrix} \Bigr).
  \]
  There also exist constants $\tilde{M}_1,\tilde{M}_2>0$
  independent of $n$ and $\tilde{B}$ such that for $n$ large
  enough $(A^n-I)\tilde{B}$ is contained in a parallelogram
  $\tilde{E}_{n,2}$ 
  with centre $x_n=(A^n-I)x_{\tilde{B}}$ and the vertices
  \begin{align*}
    &x_n +\tilde{r}\Bigl( \tilde{M}_1(\lambda_2^n-1) \begin{bmatrix} 1 \\ \gamma
    \end{bmatrix} \pm \tilde{M}_2(1-\lambda_1^n) \begin{bmatrix}
      0 \\ 1
    \end{bmatrix} \Bigr), \\
    &x_n
      -\tilde{r}\Bigl(\tilde{M}_1(\lambda_2^n-1)\begin{bmatrix} 1
                                                  \\ \gamma
                                                \end{bmatrix} \pm
                                                \tilde{M}_2(1-\lambda_1^n) \begin{bmatrix}
                                                  0 \\
                                                  1 \end{bmatrix}
                                                \Bigr).
  \end{align*}
  Such parallelograms are illustrated in
  Figure~\ref{fig:anotherparallelogram}.
  
  \begin{figure}
    \begin{center}
      \includegraphics[scale=0.18]{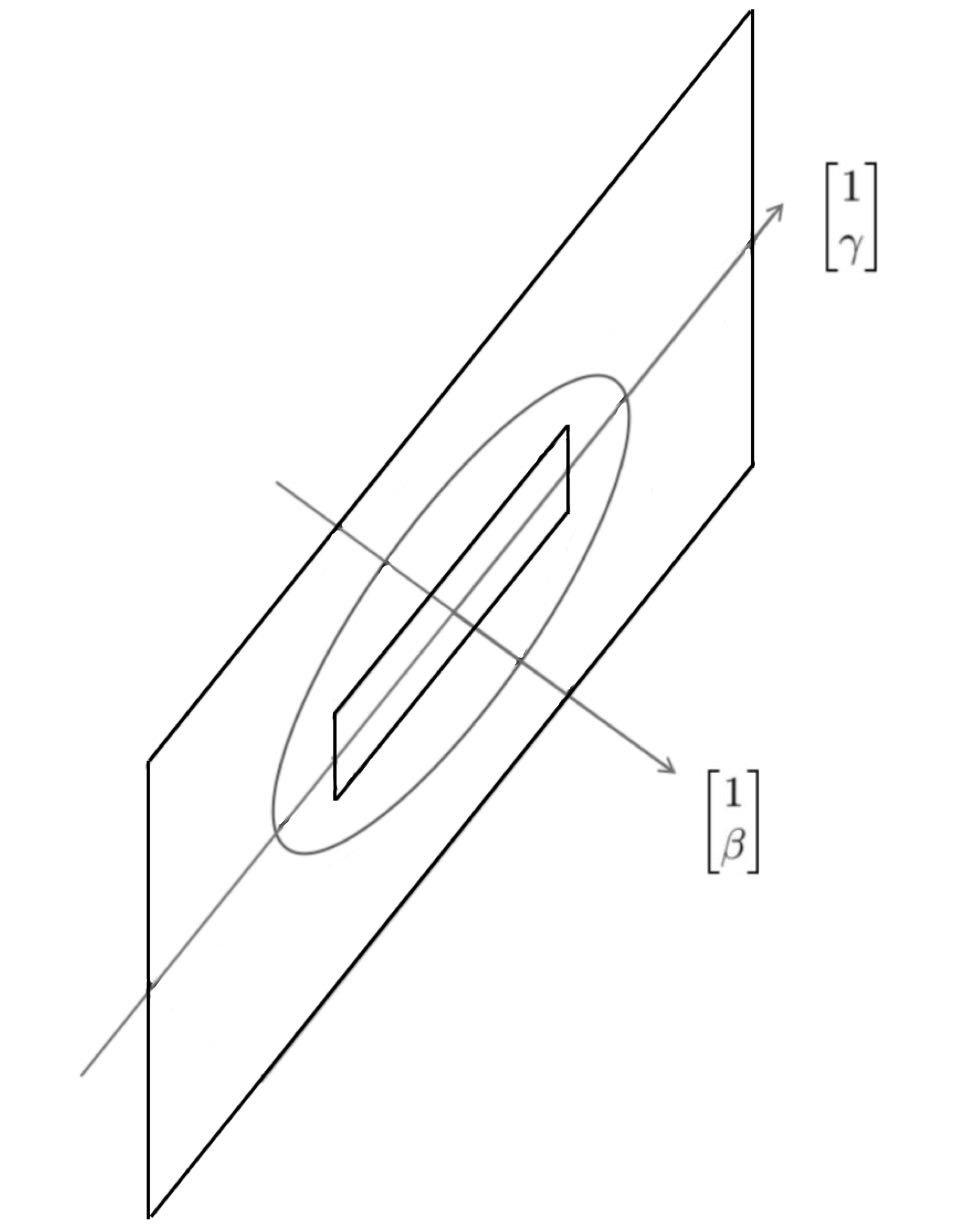}
    \end{center}
    \caption{The ellipse $R_{n,i}$ and parallelograms
      $\tilde{E}_{n,1}$ and $\tilde{E}_{n,2}$.}
    \label{fig:anotherparallelogram}
  \end{figure}
    
  We first give an upper bound of
  $\#((A^n-I)\tilde{B})\cap \mathbb{Z}^2$.  For simplicity,
  assume that
  $x_n=0$. 
  Note that
  \[
    \#(A^n-I)\tilde{B}\cap \mathbb{Z}^2\ge \#\tilde{E}_{n,1}\cap
    \mathbb{Z}^2,
  \]
  and for $z=\begin{bmatrix}z_1, & z_2\end{bmatrix}\in
  \mathbb{Z}^2\cap \tilde{E}_{n,1}$, we have
  \[
    z_1\in
    \Bigl(-\frac{M_1}{\sqrt{1+\gamma^2}}(\lambda_2^n-1)\tilde{r},
    \frac{M_1}{\sqrt{1+\gamma^2}}(\lambda_2^n-1)\tilde{r}\Bigr),
  \]
  and
  \[
    |\gamma z_1-z_2|<M_2(1-\lambda_1^n) \tilde{r}.
  \]
  Since $\lambda_1^n<\frac{1}{2}$ holds for large $n$,
  \[
    \#\tilde{E}_{n,1}\cap \mathbb{Z}^2\ge \#\Bigl\{z_1\in
    \Big[1,\frac{M_1}{\sqrt{1+\gamma^2}}(\lambda_2^n-1)\tilde{r}\Bigr)\cap
    \mathbb{Z} : \{\gamma
    z_1\}\in\big[0,\frac{1}{2}M_2\tilde{r}\bigr) \Bigr\}.
  \]
  Since
  $\gamma$ is irrational, by Theorem~\ref{udmo} the sequence
  $(\gamma
  z_1)_{z_1=1}^\infty$ is uniformly distributed modulo 1.  Taking
  the notation of Definition 2.2 in mind, with $N= \lfloor
  \frac{M_1}{\sqrt{1+\gamma^2}}(\lambda_2^n-1)\tilde{r}\rfloor$
  and
  $[a,b)=\big[0,\frac{1}{2}M_2\tilde{r}\bigr)$, then for large
  $n$,
  \begin{align*}
    \#\Bigl\{\, z_1\in
    & \Big[1,\frac{M_1}{\sqrt{1+\gamma^2}}
      (\lambda_2^n - 1) \tilde{r}\Bigr)\cap \mathbb{Z} :
      \{\gamma z_1\}\in\big[0,\frac{1}{2}M_2\tilde{r}\bigr)
      \,\Bigr\}\\
    &= A \biggl(\Bigl[ 0,\frac{1}{2}M_2\tilde{r} \Bigr), \biggl\lfloor
      \frac{M_1}{\sqrt{1+\gamma^2}} (\lambda_2^n-1) \tilde{r}
      \biggr\rfloor, (\{\gamma z_1\}) \biggr)\\
    & \ge
      \frac{M_1M_2}{3\sqrt{1+\gamma^2}}(\lambda_2^n-1)
      \tilde{r}^2 \ge \frac{M_1M_2}{3\sqrt{1+\gamma^2}}
      \tilde{r}^2 H_n.
  \end{align*}
  Similarly, we obtain that for large $n$
  \[
    \#\tilde{E}_{n,2}\cap \mathbb{Z}^2\le
    \frac{4\tilde{M}_1\tilde{M}_2}{\sqrt{1+\gamma^2}} \tilde{r}^2
    H_n.
  \]
  This finish the proof.
\end{proof}

For a ball $B:=B(x,r)$ and $n\ge1$,
\begin{equation*}\label{measure:ball}
\begin{split}
  \mu_n(B)&=\sum_{1\le i\le H_n\atop B\cap
    E_{n,i}\ne\emptyset}\frac{\mu(B\cap E_{n,i})}{\mu(E_n)}
  \\
  & \le H_n^{-1} \#\bigl\{\, 1\le i\le H_n : B\cap
  E_{n,i}\ne\emptyset \,\bigr\}.
\end{split}
\end{equation*}
For large $n$ depending on $B$ such that
$|E_{n,i}|\le \frac{1}{4}r$, we observe that
\[
  \{\, 1\le i\le H_n : B\cap E_{n,i}\ne\emptyset \,\} \subset
  \bigl\{\, 1\le i\le H_n : x_{n,i}\in \frac{5}{4}B \,\bigr\},
\]
where $x_{n,i}$ is the centre of $E_{n,i}$, and it follows from
Lemma~\ref{number} that for large $n$
\[
  \#\{\, 1\le i\le H_n : B\cap E_{n,i}\ne\emptyset \,\} \le
  \#\mathcal{P}_n\cap \frac{5}{4}B\lesssim r^2H_n,
\]
which implies that 
\[
  \mu_n(B) \lesssim   H_n^{-1}r^2H_n= r^2.
\]
Also
\begin{equation*}\label{measure:balll}
\begin{split}
  \mu_n(B)&=\sum_{1\le i\le H_n\atop B\cap
    E_{n,i}\ne\emptyset}\frac{\mu(B\cap E_{n,i})}{\mu(E_n)} \ge
  \sum_{1\le i\le H_n\atop E_{n,i}\subset
    B}\frac{\mu(E_{n,i})}{\mu(E_n)}
  \\
  & =H_n^{-1} \#\{\, 1\le i \le H_n : E_{n,i}\subset B \,\}.
\end{split}
\end{equation*}
Since 
\[
  \bigl\{\, 1\le i\le H_n : x_{n,i}\in
  \frac{3}{4}B \,\bigr\} \subset \{\, 1\le i\le H_n :
  E_{n,i}\subset B \,\},
\]
and 
\[
  \#\mathcal{P}_n\cap \frac{3}{4}B\asymp r^2H_n,
\]
we have
\[
  \mu_n(B)\gtrsim H_n^{-1} r^2H_n= r^2.
\]
We have now proven that
$(\mu_n)_{n\ge1}$ satisfies the inequalities \eqref{munballn}.

\subsection{Further estimates on $\boldsymbol{\mu_n}$} 

In the following, we shall show that there are constants $C$ and
$s$ such that $\mu_n (B(x,r)) \leq Cr^s$ holds for any
$x\in\T^2$, $r>0$ and $n$ large enough.

Consider the parallelogram $P\subset \mathbb{R}^2$ with vertices
\begin{align*}
  & \frac{f}{\sqrt{1 + \gamma^2}} \begin{bmatrix} 1 \\ \gamma
  \end{bmatrix} \pm \frac{\ell}{\sqrt{1 + \beta^2}} \begin{bmatrix} 1 \\
    \beta \end{bmatrix} ,  \quad \text{and}\\
  -&\frac{f}{\sqrt{1 + \gamma^2}} \begin{bmatrix} 1 \\
    \gamma \end{bmatrix} \pm \frac{\ell}{\sqrt{1 +
  \beta^2}} \begin{bmatrix} 1 \\
    \beta \end{bmatrix}.
\end{align*}
Then the direction of sides of $P$ are $\begin{bmatrix} 1 & \gamma
\end{bmatrix}^T$ and $\begin{bmatrix} 1 & \beta \end{bmatrix}^T$,
and the corresponding side lengths are $2f$ and $2\ell$.

\begin{lemma}\label{claim}
  There exists a constant $c_2>0$ depending only on $A$ such that
  if $f<c_2\ell^{-1}$, then the sets $P + z_1$ and $P + z_2$ do
  not intersect each other if $z_1$ and $z_2$ are two different
  elements in $\mathbb{Z}^2$. In particular, the projection of
  $P$ to $\mathbb{T}^2$ by $P \mapsto P \pmod 1$ is injective, or
  in other words, different points in $P$ do not overlap when
  projected to $\mathbb{T}^2$.
\end{lemma}

\begin{proof}
  If there is no risk of confusion, we shall often write $P$
  instead of $P\pmod1$.  Note that
  $|P\cap (\{0\}\times \mathbb{R})|\le
  2\sqrt{1+\gamma^2}\ell$

  Consider a line with slope $\gamma$ going out from the point
  $(0,0)$. The equation of this line is $y = \gamma x$, and
  $\gamma $ is an algebraic number of degree $2$. This line will
  wrap around the torus and come close to the point
  $(0,0)$. Suppose that $p$ and $q$ are integers such that
  $|\gamma q - p| < r$ with $0<r<2\sqrt{1+\gamma^2}\ell $. Then
  $|\gamma - pq^{-1} | < q^{-1}r $, but Liouville’s theorem on
  Diophantine approximation implies that
  $|\gamma - pq^{-1} | > cq^{-2}$ where $c$ is a constant
  depending only on $\gamma$. Hence it follows that
  $q \ge c\frac{1}{2\sqrt{1+\gamma^2}}\ell^{-1}$ and that
  $|\gamma q - p| = q |\gamma - p q^{-1}| > cq^{-1} >
  2\sqrt{1+\gamma^2}\ell$.

  This implies that if $2f< \frac{c}{2}\ell^{-1}$, then the sets
  $P + z_1$ and $P + z_2$ do not intersect each other if $z_1$
  and $z_2$ are two different elements in $\mathbb{Z}^2$, and the
  parallelogram $P\pmod 1$ does not overlap
  itself. 
\end{proof}
 
Let $B:=B(x,r)\subset \T^2$. Define
\[
  \ell_{1,n}=\frac{1}{n}\log|\lambda_1^n-1|\quad 
  {\rm and }
  \quad \ell_{2,n}=\frac{1}{n}\log|\lambda_2^n-1|.
\]

In the following we consider the parallelogram with vertices
\begin{align*}
  &x+ c_1^{-1}r\Bigl(\frac{1}{\sqrt{1+\beta^2}}\begin{bmatrix}1
    \\ \beta \end{bmatrix}\pm
  \frac{1}{\sqrt{1+\gamma^2}}\begin{bmatrix}1 \\
    \gamma \end{bmatrix}\Bigr),\\
  &x-  c_1^{-1}r\Bigl(\frac{1}{\sqrt{1+\beta^2}}\begin{bmatrix}1
    \\ \beta \end{bmatrix}\pm
  \frac{1}{\sqrt{1+\gamma^2}}\begin{bmatrix}1 \\
    \gamma \end{bmatrix}\Bigr),
\end{align*}
where
$c_1=\frac{\sqrt{1+\beta^2}\sqrt{1+\gamma^2}}{|\beta-\gamma|}$.
This parallelogram contains $B$, and for simplicity we will use
$B$ to denote this parallelogram in the following arguments.

\begin{itemize}
\item Case 1: $\tau>\frac{1}{2}\log|\lambda_2|$. 

  First of all, we discuss the separation between different
  parallellograms $E_{n,i}$.

  Taking $\ell=2c_2^{-1}|\lambda_2^n-1|^{-\frac{1}{2}}$ and
  $f=|\lambda_2^n-1|^{\frac{1}{2}}$, note that
  $c_1^{-1}e^{-n\tau}<\ell$ for large $n$, then $P$ contains
  $B(0,e^{-\tau n})$.
  Using Lemma~\ref{claim}, the parallelogram $P$ does not overlap itself,
  and moreover $P+z_1$ and $P+z_2$ don't intersect for different
  $z_1$ and $z_2$ in $\mathbb{Z}^2$. Hence
  $\{\, (A^n-I)^{-1}(P+z) : z\in\mathbb{Z}^2 \,\}$ do not
  intersect each other. Note that for any $i$, there exists a
  unique $z$ such that $R_{n,i}\subset (A^n-I)^{-1}(P+z)$. We
  observe that the side-lengths of $(A^n-I)^{-1}(P+z)$ are
  $2\ell
  (1-\lambda_1^n)^{-1}=2c_2^{-1}e^{-n(\frac{1}{2}\ell_{2,n}+\ell_{1,n})}$
  and $2e^{-\frac{n}{2}\ell_{2,n}}$, which implies that we have
  the separation $\asymp|\lambda_2^n-1|^{-\frac{1}{2}}$ in the
  direction $\begin{bmatrix} 1 & \gamma \end{bmatrix}^T$ and
  $\asymp|\lambda_2^n-1|^{-\frac{1}{2}}$ in the direction
  $\begin{bmatrix} 1 & \beta \end{bmatrix}^T$ between the
  parallelograms $\{E_{n,i}\}_i$. This means that
  $(A^n-I)^{-1}(P+z)$ is contained in a parallelogram with
  side-lengths $\asymp|\lambda_2^n-1|^{-\frac{1}{2}}$ in the
  direction $\begin{bmatrix} 1 & \gamma \end{bmatrix}^T$ and
  $\asymp|\lambda_2^n-1|^{-\frac{1}{2}}$ in the direction
  $\begin{bmatrix} 1 & \beta \end{bmatrix}^T$, and for two
  different $z$, these two parallellograms do not intersect.
  
  
  
\begin{itemize}
\item [(i)] If
  $r\le \frac{c_1}{\sqrt{1+\beta^2}}e^{-n(\tau +\ell_{2,n})}$,
  then $B$ intersects at most one parallelogram among
  $\{E_{n,i}\}_i$ and it can be completely contained in one of
  them. Hence
  \[
    \mu_n(B)\lesssim e^{2n\tau}r^2\lesssim
    r^{\frac{2\ell_{2,n}}{\tau+\ell_{2,n}}}.
  \]
  
\item [(ii)] If
  $ \frac{c_1}{\sqrt{1+\beta^2}}e^{-n(\tau +\ell_{2,n})}<r\le
  \frac{c_1}{\sqrt{1+\gamma^2}}e^{- n(\tau +\ell_{1,n})}$, in
  this case, $B$ intersects at most one parallelogram, but cannot
  be contained in one of them. Then
  \[
    \mu_n(B)\lesssim e^{2n\tau}re^{- n(\tau
      +\ell_{2,n})}=re^{n(\tau -\ell_{2,n})}.
  \]
  When $\tau\ge \ell_{2,n}$, by
  $r\lesssim e^{- n(\tau +\ell_{1,n})}$, we have
  \[
    \mu_n(B)\lesssim r^{1 - \frac{\tau
        -\ell_{2,n}}{\tau+\ell_{1,n}}} = r^{\frac{\ell_{1,n} +
        \ell_{2,n}}{\tau+\ell_{1,n}}}.
  \]
  When $\tau< \ell_{2,n}$, by
  $e^{- n(\tau +\ell_{2,n})}\lesssim r$, we have
  \[
    \mu_n(B) \lesssim r^{1-\frac{\tau -\ell_{2,n}}{\tau +
        \ell_{2,n}}} = r^{\frac{2\ell_{2,n}}{\tau + \ell_{2,n}}}.
  \]
  Combining these sub-cases, we conclude that 
  \[
    \mu_n(B)\lesssim
    r^{\min\bigl\{\frac{\ell_{1,n}+\ell_{2,n}}{\tau+\ell_{1,n}},\frac{2\ell_{2,n}}{\tau+\ell_{2,n}}\bigr\}}.
  \]

\item [(iii)] If
  $\frac{c_1}{\sqrt{1+\gamma^2}}e^{- n(\tau +\ell_{1,n})}\le
  r<2c_2^{-1}e^{-n(\frac{1}{2}\ell_{2,n}+\ell_{1,n})}$, then
  \[
    \#\{\, 1\le i\le H_n : B\cap E_{n,i}\ne\emptyset \,\}
    \lesssim 1,
  \] 
  which implies that
  \begin{equation*}
    \begin{split}
      \mu_n(B)&\lesssim e^{2n\tau}e^{- n(\tau +\ell_{1,n})} e^{- n(\tau +\ell_{2,n})}=e^{- n(\ell_{1,n} +\ell_{2,n})}\\
      &\lesssim r^{\frac{\ell_{1,n} +\ell_{2,n}}{\tau
          +\ell_{1,n}}}.
    \end{split}
  \end{equation*}

\item [(iv)] If
  $r\ge 2c_2^{-1}e^{-n(\frac{1}{2}\ell_{2,n}+\ell_{1,n})}$, in
  this case, $B$ intersects
  \[
    \lesssim \frac{r^2}{e^{-n(\frac{1}{2} \ell_{2,n} +
        \ell_{1,n})} e^{-\frac{n}{2}\ell_{2,n}}}
  \]
  parallelograms, and we get
  \begin{equation*}
    \begin{split}
      \mu_n(B)&\lesssim  e^{2n\tau}\frac{r^2}{e^{-n(\ell_{2,n}+\ell_{1,n})}}e^{- n(\tau +\ell_{1,n})} e^{- n(\tau +\ell_{2,n})}\\
      &= r^2.
    \end{split}
  \end{equation*}

\end{itemize}
Combining (i)--(iv), for $\epsilon >0$, we have
\[
  \mu_n(B)\lesssim r^{\min\{\frac{\ell_{1,n} +\ell_{2,n}}{\tau
      +\ell_{1,n}},
    \frac{2\ell_{2,n}}{\tau+\ell_{2,n}}\}}<r^{s_1-\epsilon},
\]
where
$s_1=\min\{\frac{\log|\lambda_2|}{\tau },
\frac{2\log|\lambda_2|}{\tau+\log|\lambda_2|}\}$. Letting
$\epsilon\to0$, it follows in this case that
\[
  \dimh R_\tau\ge s_1.
\]

\item Case 2: $\tau<\frac{1}{2}\log|\lambda_2|$.
  \label{eq:case2}
  
  Taking $\ell=\frac{1}{3}$ and $f=\min\{\frac{1}{3}, 3c_2\}$,
  $B(0,e^{-n\tau})$ is contained in $P$, which gives the
  separation $\asymp e^{-n\ell_{1,n}}$ in the direction
  $\begin{bmatrix} 1 & \beta \end{bmatrix}^T$ and
  $\asymp e^{-n\ell_{2,n}}$ in the direction
  $\begin{bmatrix} 1 & \gamma \end{bmatrix}^T$. (Separation means
  the same thing as in Case~1.) It implies that for any ball, the
  number of parallelograms intersecting the ball in the direction
  $\begin{bmatrix} 1 & \beta \end{bmatrix}^T$ is uniformly
  bounded, which means that it suffices to consider the direction
  $\begin{bmatrix} 1 & \gamma \end{bmatrix}^T$.

  Now put $\ell=c_1^{-1}e^{-\tau n}$ and $f=c_3e^{\tau n}$, where
  $c_3<c_2c_1$. Then $P$ contains $B(0,e^{-n\tau})$. Hence
  $E_{n,i}$ is contained in the parallelogram with side-lengths
  $c_1^{-1}e^{-n(\tau+\ell_{1,n})}$ in the direction
  $\begin{bmatrix} 1 & \beta \end{bmatrix}^T$ and
  $c_3e^{n(\tau-\ell_{2,n})}$ in the direction
  $\begin{bmatrix} 1 & \gamma \end{bmatrix}^T$.



  \begin{itemize}
  \item [(i)]
    $r\le \frac{c_1}{\sqrt{1+\beta^2}}e^{-n(\tau
      +\ell_{2,n})}$. In this case, $B$ intersects at most one
    parallelogram in the direction
    $\begin{bmatrix} 1 & \gamma \end{bmatrix}^T$. Then
    \[
      \mu_n(B)\lesssim e^{2\tau n}r^2\lesssim
      r^{\frac{2\ell_{2,n}}{\tau+\ell_{2,n}}}.
    \]
  \item [(ii)]
    $ \frac{c_1}{\sqrt{1+\beta^2}}e^{-n(\tau +\ell_{2,n})}< r\le
    c_3e^{n(\tau-\ell_{2,n})}$. In this case, $B$ intersects at
    most one parallelogram in the direction
    $\begin{bmatrix} 1 & \gamma \end{bmatrix}^T$.  Then
    \[
      \mu_n(B)\lesssim e^{2\tau
        n}re^{-n(\tau+\ell_{2,n})}=re^{n(\tau-\ell_{2,n})}\lesssim
      r^{\frac{2\ell_{2,n}}{\tau+\ell_{2,n}}}.
    \]
  \item [(iii)] $r>c_3e^{n(\tau-\ell_{2,n})}$. In this case, $B$
    intersects at most
    \[
      \frac{r}{e^{n(\tau-\ell_{2,n})}}
    \]
    parallelograms. The intersection of $B$ and an ellipse is
    contained in a rectangle with side-lengths
    $e^{-n(\tau+\ell_{2,n})}$ and
    $\min\{r, e^{-n(\tau+\ell_{1,n})}\}$. Then
    \begin{equation*}
      \begin{split}
        \mu_n(B)&\lesssim  \frac{r}{e^{n(\tau-\ell_{2,n})}}e^{2\tau n}e^{-n(\tau+\ell_{2,n})}\min\{r, e^{-n(\tau+\ell_{1,n})}\}\\
        &<r\min\{r, e^{-n(\tau+\ell_{1,n})}\}\le r^2.
      \end{split}
    \end{equation*}
  \end{itemize}
  Combining (i)--(iii), for $\epsilon >0$, we have
  \[
    \mu_n(B)\lesssim r^{\frac{2\ell_{2,n}}{\tau+\ell_{2,n}}} <
    r^{\frac{2 \log|\lambda_2|}{\tau + \log|\lambda_2|}-\epsilon
    }.
  \]
  Letting $\epsilon\to0$, we get
  \[
    \dimh R_\tau\ge
    \frac{2\log|\lambda_2|}{\tau+\log|\lambda_2|}.
  \]

\item Case 3:   $\tau=\frac{1}{2}\log|\lambda_2|$.

  For $\tilde{\tau}>\tau$, by Case 1 and
  $R_\tau\supset R_{\tilde{\tau}}$, we have
  \[
    \dimh R_\tau\ge \dimh R_{\tilde{\tau}}\ge
    \frac{2\log|\lambda_2|}{\tilde{\tau}+\log|\lambda_2|}.
  \]
  Letting $\tilde{\tau}$ tend to $\tau$, we conclude that
  \[
    \dimh R_\tau\ge
    \frac{2\log|\lambda_2|}{\tau+\log|\lambda_2|}.
  \]
\end{itemize}
Therefore we conclude from Cases 1--3 that  for $\tau>0$, 
\[
  \dimh R_\tau \ge \min\Bigl\{\frac{\log|\lambda_2|}{\tau },
  \frac{2\log|\lambda_2|}{\tau+\log|\lambda_2|}\Bigr\}.
\]

\section{Proof to the Example~\ref{ex:firstexample}}
\label{sec:example}

Put $R_n=\{x\in\T^3: A^nx\pmod 1\in B(x,e^{-n\tau})\}.$ 

Note that
\[
  \begin{bmatrix}1\\0\\0\end{bmatrix}, \quad \begin{bmatrix}
    0\\1\\\frac{\lambda-a}{b} \end{bmatrix},
  \quad \begin{bmatrix}
    0\\1\\\frac{\lambda^{-1}-a}{b}\end{bmatrix}
\]
are eigenvectors with eigenvalues $m$, $\lambda$ and
$\lambda^{-1}$, denoted by $\beta_1,\beta_2,\beta_3$
respectively.

Write
\[
  P_n = \Bigl\{\, \Bigl( \frac{i_1}{m^n-1}, \frac{i_2}{S_n},
  \frac{i_3}{S_n} \Bigr) : 0\le i_1<m^n-1,0\le
  i_j<S_n,j=1,2 \,\Bigr\},
\]
where 
\begin{equation}
S_n=
\begin{cases}
\sum_{j=-k}^k\lambda^j \quad &n=2k+1,\\
\sqrt{(a+d)^2-4}(\lambda^k-\lambda^{-k})\quad &n=2k.
\end{cases}
\end{equation}
Note that $S_n\asymp \lambda^{-n/2}$ for large $n$.

Then by Lemmata~2.4--2.5 and Remark 2.7 in \cite{hupe},
\begin{itemize}
\item for $n=2k+1$, $P_n$ consists of all periodic points with
  period $n$.
\item for $n=2k$, $P_n$ contains all periodic points with period
  $n$.
\end{itemize}

For $n\ge 1$, denote
\[
  r_{n,1}=e^{-n\tau}(1-\lambda^{-n})^{-1},\quad
  r_{n,2}=e^{-n\tau}(\lambda^n-1)^{-1},\quad
  r_{n,3}=e^{-n\tau}(m^{n}-1)^{-1}.
\]
Since we have assumed that $m>\lambda^{1/2}$, then we conclude
that
\[
  (m^n-1)^{-1}<S_n
\]
for large $n$.

With these in mind, we prove the dimension formula in the
example, and the proof is divided in two parts.

\subsection{The upper bound}

The proof relies on finding an efficient covering by balls of the
limsup set $R_\tau$, and notice that
$R_\tau\subset \bigcup_{n\ge N}R_n$, $N\ge1$.  We shall use balls
of radius $r_{n,k}$ to cover $R_n$. We now estimate the number of
such balls and corresponding Hausdorff dimension.

\begin{itemize}
\item[(i)] {\bf Case $\boldsymbol{k=1}$}.
  
  \begin{itemize}
  \item If $\tau<\log m$, the number of such balls is
    \begin{multline*}
      \lesssim (m^n-1) (\lambda^n-1) (1-\lambda^{-n})
      \Bigl(\frac{r_{n,1}}{(m^n-1)^{-1}} \Bigr)^{-1}\\
      = e^{n\tau}
      (\lambda^n-1) (1 - \lambda^{-n})^2.
    \end{multline*}
    It  gives 
    \[
      \dimh R_\tau\le \frac{\tau+\log\lambda}{\tau}.
    \]
  \item If $\tau\ge\log m$, the number of such balls is
    \[
      \lesssim (m^n-1) (\lambda^n-1) (1-\lambda^{-n}).
    \]
    It follows that 
    \[
      \dimh R_\tau\le \frac{\log m+\log\lambda}{\tau}.
    \]
  \end{itemize}

\item[(ii)] {\bf Case $\boldsymbol{k=2}$}.
  
  \begin{itemize}
  \item If $m\le \lambda$, then for any $\tau$, we have
    \[
      \tau+\log\lambda>\log m \quad {\rm and}\quad r_{n,2}\le
      r_{n,3} \ (<r_{n,1}),
    \]
    and for large $n$, 
    \[
      r_{n,2}<(m^n-1)^{-1}.
    \]
    The number of covering balls of radius $r_{n,2}$ is
    \[
      \lesssim (m^n-1) (\lambda^n-1) (1-\lambda^{-n})
      \Bigl(\frac{r_{n,1}}{r_{n,2}}\Bigr)
      \Bigl(\frac{r_{n,3}}{r_{n,2}}\Bigr) = (\lambda^n-1)^3.
    \]
    In this case, we get an estimate 
    \[
      \dimh R_\tau\le \frac{3\log\lambda}{\tau+\log\lambda}.
    \]
  \item If $\lambda<m$, we have $r_{n,3}<r_{n,2}\ (<r_{n,1})$, then
    there are two sub-cases.
    \begin{itemize}
    \item If $\tau<\log m-\log\lambda$, then for large $n$
      \[
        r_{n,2}>(m^n-1)^{-1}.
      \]
      Hence the number of such balls is
      \begin{multline*}
        \lesssim
        (m^n-1) (\lambda^n-1) (1-\lambda^{-n})
        \Bigl(\frac{r_{n,1}}{r_{n,2}}\Bigr)
        \Bigl(\frac{r_{n,2}}{(m^n-1)^{-1}}\Bigr)^{-1} \\
        = (\lambda^n-1)^3 e^{n\tau}.
      \end{multline*}
      It follows that 
      \[
        \dimh R_\tau\le \frac{\tau+3\log\lambda}{\tau +
          \log\lambda}.
      \]
      
    \item If $\tau>\log m-\log\lambda$, then for $n$ large
      \[
        r_{n,2} < (m^n-1)^{-1}.
      \]
      Hence the number of such balls is
      \[
        \lesssim (m^n-1) (\lambda^n-1) (1-\lambda^{-n})
        \Bigl(\frac{r_{n,1}}{r_{n,2}}\Bigr) =
        (\lambda^n-1)^2(m^n-1).
      \]
      We have
      \[
        \dimh R_\tau\le \frac{2\log\lambda+\log
          m}{\tau+\log\lambda}.
      \]
    \end{itemize}
  \end{itemize}

\item[(iii)] {\bf Case $\boldsymbol{k=3}$}.
  
  \begin{itemize}
  \item If $m\le \lambda$, we have
    $r_{n,2}\le r_{n,3}\ r_{n,1})$. The number of covering
    balls is
    \[
      \lesssim (m^n-1) (\lambda^n-1) (1-\lambda^{-n})
      \Bigl(\frac{r_{n,1}}{r_{n,3}}\Bigr) = (m^n-1)^2
      (\lambda^n-1).
    \]
    It  gives 
    \[
      \dimh R_\tau\le \frac{2\log m+\log\lambda}{\tau+\log m}.
    \]
    
  \item If $\lambda<m$, we have $r_{n,3}<r_{n,2}\ (<r_{n,1})$.
    The number of the covering balls is
    \[
      \lesssim (m^n-1) (\lambda^n-1) (1-\lambda^{-n}) \Bigl(
      \frac{r_{n,1}}{r_{n,3}}\Bigr)
      \Bigl(\frac{r_{n,2}}{r_{n,3}}\Bigr) = (m^n-1)^3.
    \]
    Then 
    \[
      \dimh R_\tau\le \frac{3\log  m}{\tau+\log m}.
    \]
  \end{itemize}
\end{itemize}
    
We conclude that 
\begin{itemize}
\item if $m\le \lambda$, 
  \[
    \dimh R_\tau\le
    \min\Bigl\{\frac{3\log\lambda}{\tau+\log\lambda},\frac{2\log
      m+\log\lambda}{\tau+\log
      m},\frac{\tau+\log\lambda}{\tau},\frac{\log
      m+\log\lambda}{\tau}\Bigr\}.
  \]
      
\item if $m> \lambda$, 
  \begin{multline*}
    \dimh R_\tau\le
    \min\Bigl\{\frac{\tau+3\log\lambda}{\tau+\log\lambda},\frac{3\log
      m}{\tau+\log m},\frac{2\log\lambda+\log
      m}{\tau+\log\lambda},\\
    \frac{\tau+\log\lambda}{\tau},\frac{\log
      m+\log\lambda}{\tau}\Bigr\}.
  \end{multline*}
\end{itemize}
      
\subsection{The lower bound}

Define 
\[
  \mu_n=\frac{1}{\mu(R_n)}\mu|_{R_n},
\]
where $\mu$ denotes the Lebesgue measure restricted on
$\T^3$. Since the periodic points are well-distributed,
$(\mu_n)_n$ satisfies the inequalities \eqref{munballn}. In the
following, we estimate $\mu_n(B)$ for any ball $B:=B(x,r)$ and
any $n$.

Recall that we have assumed that $m>\lambda^{1/2}$ which implies
that $\log m>\frac{1}{2}\log\lambda$.

Assume that $R_n=\bigcup_{i}R_{n,i}$. For $n\ge1$, 
\[
  \mu_n(B)\le (\mu(R_n))^{-1}\mu(B\cap R_n)\le
  (\mu(R_n))^{-1}\sum_{i\colon R_{n,i}\cap
    B\ne\emptyset}\mu(B\cap R_{n,i}).
\]

Note that when $\tau>\frac{1}{2}\log\lambda$, the distribution of
periodic points gives that the separation of the ellipsoids contained
in $R_n$ is $\asymp (m^n-1)^{-1}$ in the direction $\beta_1$, and
$\asymp S_n^{-1}$ in both directions $\beta_2$ and $\beta_3$.

If $r\le (m^n-1)^{-1}$, then $B$ intersects  at most $1$ ellipsoid;
If  $(m^n-1)^{-1}<r\le S_n^{-1},$ then $B$ intersects  
\[
  \lesssim \frac{r}{(m^n-1)^{-1}}
\]
ellipsoids;
If $S_n^{-1}<r$, then $B$ intersects 
\[
  \lesssim
  \frac{r^3}{(m^n-1)^{-1}(\lambda^n-1)^{-1}(1-\lambda^{-n})^{-1}}
\]
ellipsoids.
 
When $\tau\le \frac{1}{2}\log\lambda$, we need to consider the
distance among those ellipsoids. The separation is
$\asymp (m^n-1)^{-1}$ in the direction $\beta_1$. Recalling
Case~2 on Page~\pageref{eq:case2}, the separation is $\asymp 1$ in
the direction $\beta_3$ and $\asymp e^{n(\tau-\log\lambda)}$ in
the direction $\beta_2$. If $r\le e^{n(\tau-\log\lambda)},$
then $B$ intersects
\[
  \lesssim \frac{r}{\min\{r,(m^n-1)^{-1}\}}
\]
ellipsoids; If $e^{n(\tau-\log\lambda)}<r$, then $B$ intersects
\[
  \lesssim \frac{r^2}{(m^n-1)^{-1}e^{n(\tau-\log\lambda)}}
\]
ellipsoids.

{\bf Dealing with  Case (i), $\boldsymbol{m> \lambda}$.}

\begin{enumerate}
\item Case $m>\lambda^{3/2}$
  \begin{itemize}
  \item When $\tau>\log m$, in this case, for large $n$, we have
    \[
      r_{n,3}<r_{n,2}<r_{n,1}<(m^n-1)^{-1}<S_n^{-1}.
    \]
    \begin{itemize}
    \item If $r\le r_{n,3}$, then
      \[
        \mu_n(B)\lesssim e^{3n\tau} r^3<r^{\frac{3\log
            m}{\tau+\log m}}.
      \]
    \item If $r_{n,3}<r\le r_{n,2}$, then we have
      \[
        \mu_n(B)\lesssim e^{3n\tau} r^2r_{n,3}\asymp
        r^2e^{n(2\tau-\log m)}<r^{\frac{2\log\lambda+\log
            m}{\tau+\log\lambda}}.
      \]
      The last inequality follows from $\tau>\log m$ and
      $e^{n}\le r^{-\frac{1}{\tau+\log\lambda}}$.
    \item If $r_{n,2}<r\le r_{n,1}$, then
      \[
        \mu_n(B)\lesssim e^{3n\tau} rr_{n,2}r_{n,3}\asymp
        re^{n(\tau-\log
          m-\log\lambda)}
        .
      \]
      For $\tau>\log m+\log \lambda$, using
      $e^n<r^{-\frac{1}{\tau}}$, we get
      \[
        \mu_n(B)<r^{\frac{\log m+\log\lambda}{\tau}}.
      \]
      For $\tau\le \log m+\log \lambda$, using
      $e^{-n}<r^{-\frac{1}{\tau+\log\lambda}}$, we have
      \[
        \mu_n(B)<r^{\frac{2\log\lambda+\log
            m}{\tau+\log\lambda}}.
      \]
      It follows from the last two inequalities that
      \[
        \mu_n(B)<r^{\min\{\frac{2\log\lambda+\log
            m}{\tau+\log\lambda},\frac{\log
            m+\log\lambda}{\tau}\}}.
      \]
       
    \item If $r_{n,1}<r\le (m^n-1)^{-1}$,
      \begin{equation*}
        \begin{split}
          \mu_n(B)&\lesssim e^{3n\tau} r_{n,1}r_{n,2}r_{n,3}=((\lambda^n-1) (m^n-1) (1-\lambda^{-n}))^{-1}\\
          &<r^{\frac{\log m+\log\lambda}{\tau}}.
        \end{split}
      \end{equation*}
      
    \item If $(m^n-1)^{-1}<r\le S_n^{-1}$, then we use
      $r_{n,1}<r$
      and get that
      \[
        \mu_n(B)\lesssim e^{3n\tau}
        r_{n,1}r_{n,2}r_{n,3}\frac{r}{(m^n-1)^{-1}}\asymp
        r(\lambda^n-1)^{-1}<r^{\frac{\tau+\log\lambda}{\tau}}.
      \]
      
    \item If $r>S_n^{-1}$,
      \[
        \mu_n(B)\lesssim e^{3n\tau} r_{n,1}r_{n,2}r_{n,3}
        \frac{r^3}{(m^n-1)^{-1} (\lambda^n-1)^{-1}
          (1-\lambda^{-n})^{-1}} = r^3.
      \]
    \end{itemize}

    By combining these estimates with Lemma~\ref{lip}, when
    $\tau>\log m$, we get that
    \[
      \dimh R_\tau\ge
      \min\Bigl\{\frac{3\log m}{\tau+\log m}
      ,\frac{2\log\lambda+\log
        m}{\tau+\log\lambda},\frac{\tau+\log\lambda}{\tau},
      \frac{\log m+\log\lambda}{\tau}\Bigr\}.
    \]
     
  \item When $\log m-\log\lambda<\tau\le\log m$, for large $n$,
    we have
    \[
      r_{n,3}<r_{n,2}<(m^n-1)^{-1}<r_{n,1}<S_n^{-1}.
    \]
    \begin{itemize}
    \item If $r\le (m^n-1)^{-1}$, similarly to the previous case,
      we get
      \begin{equation*}
        \begin{split}
          \mu_n(B)&\lesssim e^{3n\tau}r\min\{r,r_{n,3}\}\min\{r,r_{n,2}\}\\
          &<r^{\min\{\frac{3\log m}{\tau+\log m} ,
            \frac{2\log\lambda+\log m}{\tau+\log\lambda} \}}.
        \end{split}
      \end{equation*}
    \item If $(m^n-1)^{-1}<r \le r_{n,1}$, then
      \[
        \mu_n(B)\lesssim e^{3n\tau}
        rr_{n,3}r_{n,2}\frac{r}{(m^n-1)^{-1}}\asymp
        r^2e^{n(\tau-\log\lambda)}
        .
      \]
      For $\tau>\log\lambda$, we use $e^n<r^{-\frac{1}{\tau}}$,
      and conclude that
      \[
        \mu_n(B)<r^{\frac{\tau+\log\lambda}{\tau}}.
      \]
      For $\tau\le \log \lambda$, using
      $e^{-n}<r^{-\frac{1}{\tau+\log\lambda}}$, we get
      \[
        \mu_n(B)<r^{\frac{\tau+3\log\lambda}{\tau+\log\lambda}}.
      \]
      Therefore
      \[
        \mu_n(B)\lesssim
        r^{\min\{\frac{\tau+3\log\lambda}{\tau+\log\lambda},
          \frac{\tau+\log\lambda}{\tau}\}}.
      \]
      
    \item If $r_{n,1}<r \le S_n^{-1}$, we obtain
      \[
        \mu_n(B)\lesssim e^{3n\tau}
        r_{n,1}r_{n,2}r_{n,3}\frac{r}{(m^n-1)^{-1}}\asymp
        e^{-n\log\lambda}<r^{\frac{\tau+\log\lambda}{\tau}}.
      \]
       
    \item If $S_n^{-1}<r $, then
      \[
        \mu_n(B)<r^3.
      \]
    \end{itemize}
    We conclude from these cases that
    \[
      \dimh R_\tau\ge
      \min\Bigl\{\frac{\tau+3\log\lambda}{\tau+\log\lambda},
      \frac{3\log m}{\tau+\log m} , \frac{2\log\lambda+\log
        m}{\tau+\log\lambda},\frac{\tau+\log\lambda}{\tau}\Bigr\}.
    \]

  \item When $\frac{1}{2}\log\lambda<\tau\le \log m-\log\lambda$,
    here for large $n$, we have
    \[
      r_{n,3}<(m^n-1)^{-1}<r_{n,2}<r_{n,1}<S_n^{-1}.
    \]
    \begin{itemize}
    \item If $r\le (m^n-1)^{-1}$, then
      \[ \mu_n(B)\lesssim e^{3n\tau} r^2\min\{r,
        r_{n,3}\}<r^{\min\{\frac{3\log m}{\tau+\log m}
          ,\frac{2\log\lambda+\log m}{\tau+\log\lambda} \}}.
      \]
    \item If $(m^n-1)^{-1}<r\le S_n^{-1}$, we get
      \begin{equation*}
        \begin{split}
          \mu_n(B)&\lesssim e^{3n\tau} r_{n,3}\min\{r,r_{n,2}\}\min\{r,r_{n,1}\}\frac{r}{(m^n-1)^{-1}}\\
          &<r^{\min\{\frac{\tau+3\log\lambda}{\tau+\log\lambda},\frac{2\log\lambda+\log
              m}{\tau+\log\lambda} ,\frac{\tau+\log\lambda}{\tau}
            \}}.
        \end{split}
      \end{equation*}
    \item If $r>S_n^{-1}$, then
      \[
        \mu_n(B)<r^3.
      \]
       
    \end{itemize}
    Therefore
    \[
      \dimh R_\tau\ge
      \min\Bigl\{\frac{\tau+3\log\lambda}{\tau+\log\lambda},\frac{3\log
        m}{\tau+\log m} ,\frac{2\log\lambda+\log
        m}{\tau+\log\lambda} ,\frac{\tau+\log\lambda}{\tau}
      \Bigr\}.
    \]

  \item When $\tau\le
    \frac{1}{2}\log\lambda$, 
    the inequalities
    \[
      r_{n,3}<(m^n-1)^{-1}<r_{n,2}<e^{n(\tau-\log\lambda)}<r_{n,1}
    \]
    hold for large $n$.
    \begin{itemize}
    \item If $r\le e^{n(\tau-\log\lambda)}$, then
      \begin{equation*}
        \begin{split}
          \mu_n(B)&\lesssim e^{3n\tau} r\min\{r,r_{n,2}\}\min\{r,r_{n,3}\}\frac{r}{\min\{r,(m^n-1)^{-1}\}}\\
          & <r^{\min\{\frac{3\log m}{\tau+\log m},
            \frac{\tau+3\log\lambda}{\tau+\log\lambda}\}}.
        \end{split}
      \end{equation*}
    \item If $r>e^{n(\tau-\log\lambda)}$, then
      \[
        \mu_n(B)<r^3.
      \]
    \end{itemize}
    In this case, we have
    \[
      \dimh R_\tau\ge \min\Bigl\{\frac{3\log m}{\tau+\log m},
      \frac{\tau+3\log\lambda}{\tau+\log\lambda}\Bigr\}.
    \]
  \end{itemize}
     
\item Case $\lambda<m\le \lambda^{3/2}$.

  It suffices to consider the case
  $\log m-\log\lambda<\tau\le \frac{1}{2}\log\lambda$, which
  gives
  \[
    \dimh R_\tau\ge \min\Bigl\{\frac{3\log m}{\tau+\log
      m},\frac{2\log\lambda+\log m}{\tau+\log\lambda},
    \frac{\tau+3\log\lambda}{\tau+\log\lambda}\Bigr\}.
  \]
  Combining all cases, we get that
  \begin{multline*}
    \dimh R_\tau\ge \min\Bigl\{\frac{3\log m}{\tau+\log
      m}, \frac{2\log\lambda+\log m}{\tau+\log\lambda}, \\
    \frac{\tau+3\log\lambda}{\tau+\log\lambda},\frac{\log
      m+\log\lambda}{\tau},\frac{\tau+\log\lambda}{\tau}\Bigr\},
  \end{multline*}
  when $m>\lambda$.
\end{enumerate}

{\bf Dealing with  Case (ii), $\boldsymbol{m\le \lambda}$.}

\begin{itemize}
\item When $\frac{1}{2}\log\lambda<\tau\le \log m $. In this
  case, for large $n$, we have
  \[
    r_{n,2}<r_{n,3}<(m^n-1)^{-1}<r_{n,1}<S_n^{-1}.
  \]

  Therefore
  \begin{itemize}
  \item if $r\le r_{n,2}$, we have
    \[
      \mu_n(B)\lesssim
      e^{3n\tau}r^3<r^{\frac{3\log\lambda}{\tau+\log\lambda}}.
    \]
  \item if $r_{n,2}<r\le r_{n,3}$, we have
    \[
      \mu_n(B)\lesssim e^{3n\tau} r^2r_{n,2}\asymp
      r^2e^{n(2\tau-\log\lambda)}<r^{\frac{2\log
          m+\log\lambda}{\tau+\log m}},
    \]
    where the last inequality holds since
    $\tau\ge \frac{1}{2}\log\lambda$ and
    $e^n<r^{-\frac{1}{\tau+\log m}}$.
  \item if $r_{n,3}<r\le (m^n-1)^{-1}$, we derive that
    \[
      \mu_n(B)\lesssim e^{3n\tau}rr_{n,2}r_{n,3}\asymp
      re^{n(\tau-\log m-\log\lambda)}<r^{\frac{2\log
          m+\log\lambda}{\tau+\log m}}.
    \]
    Here we used that $\tau \le\log m$ and
    $e^{-n}<r^{\frac{1}{\tau+\log m}}$.
  \item if $(m^n-1)^{-1}<r \le r_{n,1} $, note that
    $\tau\le \log m\le \log\lambda$ and $r_{n,3}<r$, then
    \[
      \mu_n(B)\lesssim
      e^{3n\tau}rr_{n,2}r_{n,3}\frac{r}{(m^n-1)^{-1}}\asymp
      r^2e^{n(\tau-\log\lambda)}<r^{\frac{2\log
          m+\log\lambda}{\tau+\log m}}.
    \]
  \item if $ r_{n,1} < r\le S_n^{-1}$, we have
    \[
      \mu_n(B)\lesssim e^{3n\tau} r_{n,1}
      r_{n,2}r_{n,3}\frac{r}{(m^n-1)^{-1}} \asymp
      re^{-n\log\lambda}<r^{\frac{\tau+\log\lambda}{\tau}}.
    \]
  \item if $S_n^{-1}<r $, we get
    \[
      \mu_n(B)\lesssim e^{3n\tau}r_{n,1} r_{n,2}r_{n,3}
      \frac{r^3}{(m^n-1)^{-1}(\lambda^n-1)^{-1}(1-\lambda^{-n})^{-1}}=r^3.
    \]
  \end{itemize}
  When $\frac{1}{2}\log\lambda < \tau \le \log m$, we obtain from
  the above estimates that
  \[
    \dimh R_\tau\ge
    \min\Bigl\{\frac{3\log\lambda}{\tau+\log\lambda},\frac{2\log
      m+\log\lambda}{\tau+\log
      m},\frac{\tau+\log\lambda}{\tau}\Bigr\}.
  \]
 
\item When $\log m <\tau $, for large $n$, we have
  \[
    r_{n,2}<r_{n,3}<r_{n,1}<(m^n-1)^{-1}<S_n^{-1}.
  \]
 
  Therefore
  \begin{itemize}
  \item if $r\le r_{n,3}$, similarly to the previous one, we have
    \[
      \mu_n(B)\lesssim
      e^{3n\tau}r^2\min\{r,r_{n,2}\}<r^{\min\{\frac{3\log\lambda}{\tau+\log\lambda},\frac{2\log
          m+\log\lambda}{\tau+\log m}\}}.
    \]
  \item if $r_{n,3}<r\le r_{n,1} $, we have
    \begin{equation*}
      \begin{split}
        \mu_n(B)&\lesssim e^{3n\tau}rr_{n,2}r_{n,3}\asymp re^{n(\tau-\log m-\log\lambda)}\\
        &<\begin{cases}
          r^{\frac{2\log m+\log\lambda}{\tau+\log m}}\quad &{\rm for}\, \log m<\tau\le \log m+\log\lambda \\
          r^{\frac{\log m+\log\lambda}{\tau}}\quad &{\rm for}\, \tau> \log m+\log\lambda \end{cases}\\
        &=r^{\min\{\frac{2\log m+\log\lambda}{\tau+\log m},
          \frac{\log m+\log\lambda}{\tau}\}}.
      \end{split}
    \end{equation*}
 
  \item if $r_{n,1}<r \le (m^n-1)^{-1}$, then
    \[
      \mu_n(B)\lesssim e^{3n\tau}r_{n,1}r_{n,2}r_{n,3}\asymp
      e^{-n(\log m+\log\lambda)}<r^{\frac{\log
          m+\log\lambda}{\tau}}.
    \]
  \item if $ (m^n-1)^{-1}< r$, we have
    \[
      \mu_n(B)\lesssim e^{3n\tau} r_{n,1}
      r_{n,2}r_{n,3}\frac{r}{(m^n-1)^{-1}}\frac{r^2}{\min\{S_n^{-2},r^2\}}
      <r^{\frac{\tau+\log\lambda}{\tau}}.
    \]
  \end{itemize}
  Now we conclude that if $\tau>\log m $,
  \[
    \dimh R_\tau\ge
    \min\Bigl\{\frac{3\log\lambda}{\tau+\log\lambda},\frac{2\log
      m+\log\lambda}{\tau+\log
      m},\frac{\tau+\log\lambda}{\tau},\frac{\log
      m+\log\lambda}{\tau}\Bigr\}.
  \]

\item When $\tau \le \frac{1}{2}\log\lambda
  $.  

  Since $m>\lambda^{1/2}$,
  $\log\lambda-\log m<\frac{1}{2}\log\lambda $. There are three
  sub-cases, depending on the size of $\tau$.

  If $\tau >\log \lambda-\log m$, then for large $n$, we have
  \[
    r_{n,2}<r_{n,3}<(m^n-1)^{-1}<e^{n(\tau-\log\lambda)}<r_{n,1}.
  \]

  Depending on the size of $r$ we consider two different cases.
  \begin{itemize}
  \item If $r\le e^{n(\tau-\log\lambda)}$, we have
    \begin{equation*}
      \begin{split}
        \mu_n(B)&\lesssim e^{3n\tau}r\min\{r,r_{n,2}\}\min\{r,r_{n,3}\}\frac{r}{\min\{r,(m^n-1)^{-1}\}}\\
        &<r^{\min\{\frac{3\log\lambda}{\tau+\log\lambda},\frac{2\log
            m+\log\lambda}{\tau+\log m}\}}.
      \end{split}
    \end{equation*}
  \item In case $ e^{n(\tau-\log\lambda)}<r$, the set $B$
    intersects
    \[
      \lesssim
      \frac{r}{e^{n(\tau-\log\lambda)}}\frac{r}{(m^n-1)^{-1}}
    \]
    ellipsoids, and then
    \[
      \mu_n(B)\lesssim e^{3n\tau}r_{n,2}r_{n,3}\min\{r,r_{n,1}\}
      \frac{r}{e^{n(\tau-\log\lambda)}}\frac{r}{(m^n-1)^{-1}}<r^3.
    \]
  \end{itemize}

  If
  $\frac{1}{2}(\log \lambda-\log m)<\tau \le \log \lambda-\log
  m$, for large $n$, we have
  \[
    r_{n,2}<r_{n,3}<e^{n(\tau-\log\lambda)}\le
    (m^n-1)^{-1}<r_{n,1}.
  \]
  We consider different cases, depending on the size of $r$.
  \begin{itemize}
  \item If $r\le e^{n(\tau-\log\lambda)}$, then
    \[
      \mu_n(B)\lesssim e^{3n\tau}r\min\{r, r_{n,3}\}\min\{r,
      r_{n,2}\}<r^{\min\{\frac{3\log\lambda}{\tau+\log\lambda},\frac{2\log
          m+\log\lambda}{\tau+\log m}\}}.
    \]
   
  \item If $e^{n(\tau-\log\lambda)}<r\le (m^n-1)^{-1}$, then
    \[
      \mu_n(B)\lesssim e^{3n\tau}rr_{n,2}
      r_{n,3}\frac{r}{e^{n(\tau-\log\lambda)}}<r^{\frac{2\tau+3\log
          m}{\tau+\log m}}.
    \]
    
  \item If $(m^n-1)^{-1}<r$, then
    \begin{equation*}
      \begin{split}
        \mu_(B)&\lesssim e^{3n\tau}r_{n,2} r_{n,3} \min\{r,
        r_{n,1}\} \frac{r}{e^{n(\tau-\log\lambda)}}
        \frac{r}{(m^n-1)^{-1}}\\
        &<r^3.
      \end{split}
    \end{equation*}
  \end{itemize}
  It follows from the fact
  $\tau>\frac{1}{2}(\log \lambda-\log m)$ that
  \begin{equation*}
    \begin{split}
      \mu_(B)&< r^{\min\{\frac{3\log\lambda}{\tau+\log\lambda},
        \frac{2\log m+\log\lambda}{\tau+\log m},
        \frac{2\tau+3\log m}{\tau+\log m}\}} \\
      &=r^{\min\{\frac{3\log\lambda}{\tau+\log\lambda},\frac{2\log
          m+\log\lambda}{\tau+\log m}\}}.
    \end{split}
  \end{equation*}

  If $\tau \le \frac{1}{2}(\log \lambda-\log m)$, for large $n$,
  we have
  \[
    r_{n,2}<e^{n(\tau-\log\lambda)}\le r_{n,3}<
    (m^n-1)^{-1}<r_{n,1}.
  \]
  Again, we split into three cases depending on the size of $r$.
  \begin{itemize}
  \item If $r\le e^{n(\tau-\log\lambda)}$, then
    \[
      \mu_n(B)\lesssim e^{3n\tau}r^2\min\{r,
      r_{n,2}\}<r^{\min\{\frac{3\log\lambda}{\tau+\log\lambda},\frac{2\log
          m+\log\lambda}{\tau+\log m}\}}.
    \]
  \item If $e^{n(\tau-\log\lambda)}<r\le (m^n-1)^{-1}$, then
    \[
      \mu_n(B)\lesssim e^{3n\tau}rr_{n,2}\min\{r,
      r_{n,3}\}\frac{r}{e^{n(\tau-\log\lambda)}}<r^{\frac{2\tau+3\log
          m}{\tau+\log m}}.
    \]
  \item If $(m^n-1)^{-1}<r$, then
    \[
      \mu_n(B)<r^3.
    \]
  \end{itemize}
  Notice that when $\tau \le \frac{1}{2}(\log \lambda-\log m)$,
  \[
    \min\Bigl\{\frac{3\log\lambda}{\tau+\log\lambda},\frac{2\log
      m+\log\lambda}{\tau+\log m},\frac{2\tau+3\log m}{\tau+\log
      m},3 \Bigr\}=\frac{3\log\lambda}{\tau+\log\lambda}.
  \]

\end{itemize}
We can now conclude Case~(ii). When $m\le \lambda$,
\[
  \dimh R_\tau\ge
  \min\Bigl\{\frac{3\log\lambda}{\tau+\log\lambda},\frac{2\log
    m+\log\lambda}{\tau+\log
    m},\frac{\tau+\log\lambda}{\tau},\frac{\log
    m+\log\lambda}{\tau}\Bigr\},
\]
and we are done.

\subsection*{Acknowledgements}
Zhangnan Hu is supported by China Postdoctoral Science Foundation (Grant No. 2023M743878), by the Postdoctoral Fellowship Program of CPSF under Grant
Number GZB20240848
and  by Science Foundation of China,
University of Petroleum, Beijing (Grant No. 2462023\-SZBH\-013) .

\bibliographystyle{amsplain}

\end{document}